\documentclass[11pt,a4paper]{amsart}

\usepackage {amsfonts}
\usepackage[english]{babel}
\renewcommand\l{\lambda}
\renewcommand\L{\Lambda}
\newcommand\s{\sigma}
\newcommand\p{\varphi}

\newcommand\Th{\Theta}
\renewcommand\d{\delta}
\newcommand\De{\Delta}

\renewcommand\O{\Omega}
\newcommand\e{\varepsilon}
\newcommand\g{\gamma}
\newcommand\G{\Gamma}
\newcommand\z{\zeta}

\newcommand\R{{\mathbb R}}
\newcommand\C{{\mathbb C}}
\newcommand\N{{\mathbb N}}
\newcommand\D{{\mathbb D}}
\newcommand\T{{\mathbb T}}
\newcommand\Z{{\mathbb Z}}

\newcommand{\cF}{{\mathcal{F}}}
\newcommand{\cH}{{\mathcal{H}}}

\newcommand{\cS}{{\mathcal{S}}}

\newcommand{\cD}{{\mathcal{D}}}

\newcommand\ovl{\overline}

\newcommand\lp{\left(}
\newcommand\rp{\right)}
\renewcommand\Re{{\rm Re}}
\renewcommand\Im{{\rm Im}}

\renewcommand\arg{{\rm arg}}

\newcommand\rcoh{{\rm conv}_r(E)}

\newtheorem{theorem}{Theorem}
\newtheorem{lemma}{Lemma}
\newtheorem{corollary}{Corollary}
\newtheorem{proposition}{Proposition}
\newtheorem{definition}{Definition}

\begin{document}

\title[Blaschke-type conditions]
{Blaschke-type conditions in unbounded domains, generalized convexity and applications in perturbation theory}

\date{\today}

\author[S. Favorov]{Sergey Favorov}
\email{sfavorov@gmail.com}
\address{Mathematical School, Karazin Kharkov National University, 4 Swobody sq., 61077 Kharkov, Ukraine}

\author[L. Golinskii]{Leonid Golinskii}
\email{golinskii@ilt.kharkov.ua}
\address{Mathematics Division, Institute for Low Temperature Physics and Engineering, 47 Lenin ave., 61103 Kharkov, Ukraine}


\thanks{}
\dedicatory{}

\begin{abstract}
We introduce a new geometric characteristic of compact sets in the plane called $r$-convexity, which fits nicely into the concept
of generalized convexity and extends essentially the conventional convexity. For a class of subharmonic functions on unbounded
domains with $r$-convex compact complement, with the growth governed by the distance to the boundary, we obtain the Blaschke-type
condition for their Riesz measures. The result is applied to the study of the convergence of the discrete spectrum for the
Schatten--von Neumann perturbations of bounded linear operators in the Hilbert space.
\end{abstract}

\maketitle

\section{Introduction}

In 1915 Blaschke \cite{Bla} proved his celebrated result concerning zero sets of bounded analytic functions in the unit disk,
which became a gem of function theory. A vast literature with various refinements and far reaching extensions of the Blaschke
condition has appeared since then, see \cite{Dj, Gol, hk, Sh2, yul} and references therein.

We focus on a series of recent papers \cite{bgk, fg09, fg12, gk12}, where the authors study the zero sets of analytic functions
in the unit disk, which grow at the direction of a prescribed subset of the unit circle. The result in \cite{fg09} for analytic
functions looks as follows.

{\bf Theorem A}. Let $E\subset\partial\D$ be a closed set on the unit circle, $f$ be an analytic function in the unit disk $\D$
with the zero set $Z_f=\{z_n\}$ (each zero $z_n$ enters with its multiplicity) so that $|f(0)|=1$, and
$$ \log|f(z)|\le\frac{K_f}{{\rm dist}^q(z,E)}\,, \qquad z\in\D, \quad q>0, $$
${\rm dist}(E_1,E_2)$ is the distance between closed sets $E_1$ and $E_2$. Then for each $\e>0$
$$ \sum_n(1-|z_n|){\rm dist}^p(z_n,E)\le C(q,E,\e) K_f, \qquad p=\max(q+\kappa(E)-1+\e,0), $$
$\kappa(E)$ is the upper Minkowski dimension of $E$.

\smallskip

This result applies in perturbation theory of linear operators, although the situation there is somewhat different. The point is
that the basic objects -- the resolvent and the perturbation determinant -- are analytic functions on the resolvent set of the
corresponding operator (including infinity), which is an {\it unbounded} open set of the plane with the compact complement $E$,
the spectrum of the operator. To handle this problem, the attempts were made to go over to the unit disk, using the conformal
mapping \cite{dhk08,haka11} or the uniformization theorem \cite{gk12}, apply Theorem A and then get back by means of certain distortion
results. Such attempts were by and large successful only in the cases when $E$ is a single segment \cite{dhk08,haka11} or a finite
union of disjoint segments \cite{gk12}, and it is absolutely unclear whether it is possible to make such argument work for, say,
an arbitrary compact set on the line.

The reasoning in \cite{fg09} reveals a potential theoretic character of the problem, so the natural setting is subharmonic
functions $v$ and their Riesz measures (generalized Laplacians) $\mu=(1/2\pi)\De v$
rather than analytic functions and their zero sets. In the case $v=\log|f|$ with an analytic function $f$, the Riesz measure is
a discrete and integer-valued measure supported on $Z_f$, and $\mu\{z\}$ equals the multiplicity of the zero at $z$.

In this paper we develop a straightforward approach to the study of subharmonic functions on unbounded domains with the growth
governed by the distance to the boundary. Let $E$ be a compact set in the complex plane $\C$, which does not split the plane
(its complement $\O=\ovl\C\backslash E$ in the extended plane $\ovl\C$ is a domain, that is, a connected open in the sense of
$\ovl\C$ set). Consider a class of subharmonic on $\O$ functions subject to the following growth and normalization conditions
\begin{equation}\label{31}
v(z)\le K_v\,\psi(d(z)), \quad v(\infty)=0, \qquad d(z):={\rm dist}(z,E),
\end{equation}
$\psi$ is a positive and monotone decreasing function on $\R_+=[0,\infty)$, $\psi\to +\infty$ as $t\to 0+$
(we single out the constant $K_v$ on purpose, in view of applications in perturbation theory in Section~5). In the study
of the Riesz measures of such functions one is faced with at least two obstacles. First, the set $E$ may be small enough (polar or
just finite), so to apply the standard technique from the potential theory we step inside $\O$ and work in the ``outer neighborhood''
$$ \O_t:=\{z\in\C: \ d(z)>t\}, \quad t>0, \qquad \O=\O_0. $$
Its boundary $\partial\O_t=\{z:d(z)=t\}$ is non-polar, since it splits the plane (cf. \cite[Theorem 3.6.3]{Ran}), so the Green's
function $G_t$ for $\O_t$ exists and is unique, whenever $\O_t$ is a domain. Unfortunately, it is not hard to manufacture a set
$E$ so that $\O$ is a domain, but $\O_t$ is not for $t>0$.

To cope with this problem we introduce a new geometric characteristic -- $r$-convexity -- which fits nicely in the concept of
generalized convexity, see \cite{dgkl}. It can be defined in an arbitrary metric space, no linear structure is needed for that.
Precisely, it is well known that a closed set in $\C$ is convex if and only if it is the intersection of all closed half-planes
containing this set. For an arbitrary closed set $E$ this intersection agrees with the convex hull of $E$. As usual, we denote by $B(x,r)$,
$B^c(x,r)$, and $\partial B(x,r)$ an open unit disk of radius $r$ centered at $x$, its complement, and its
boundary, respectively
$$ B(x,r)=\{z: |z-x|<r\}, \quad B^c(x,r)=\{z:|z-x|\ge r\}, \quad \partial B(x,r)=\{z:|z-x|=r\}. $$
By replacing half-planes with exteriors of open disks $B^c$, we come to the following extension of the conventional convexity.
We start out with an obvious inclusion
\begin{equation}\label{rcon1}
E\subset {\rm conv}_r(E):=\bigcap\{B^c(z,r):\ \ E\subset B^c(z,r)\}, \quad r>0.
\end{equation}

\begin{definition}\label{d1}
We say that a closed set $E$ is $r$-convex, if $E=\rcoh$. The set $\rcoh$ is called the $r$-convex hull of $E$.
\end{definition}
In other words, $E$ is $r$-convex if
\begin{equation}\label{rcon2}
\C\backslash E=\bigcup\{B(z,r):\ \ B(z,r)\subset \C\backslash E\},
\end{equation}
that is, the complement to $E$ can be covered by open disks of a {\it fixed} radius $r>0$ which belong to this complement.
Similarly to the usual convexity, the intersection of any family of $r$-convex sets is $r$-convex. On the other hand,
in contrast to the usual convexity, a finite union of disjoint $r$-convex sets is $r'$-convex for some $r'\le r$. It is also clear
that $E_1\subset E_2$ implies ${\rm conv}_r(E_1)\subset {\rm conv}_r(E_2)$.

It follows from \eqref{rcon2}, that each $r$-convex set is also $r'$-convex for any
$r'<r$. So, the number $r_0(E):=\sup\{r: E=\rcoh\}$, called the radius of convexity of $E$, arises naturally. For instance,
each closed convex set $E$ is $r$-convex with $r_0(E)=\infty$, and it is easy to see that the same holds for each closed subset of
a line. Indeed, any open interval on a line (complementary interval of a closed set) can be covered with a disk of arbitrarily
large radius.
We show that $r_0(E)=R$ for each compact subset of a circle $\partial B(x,R)$, which contains more than two points (see
Proposition \ref{p2}). The sets with ``interior angles'', like $\{z\in\ovl\D:\ \pi/4\le\arg z\le 7\pi/4\}$, are not $r$-convex
for all $r>0$.

It turns out (see Theorem \ref{t1}) that if an $r$-convex compact set $E$ does not split the plane, then there exists $t_0=t_0(E)>0$
such that $\O_t$ is a domain for all $0\le t\le t_0$. So for such $t$
the Green's function $G_t$ for $\O_t$ exists and unique. A key potential theoretic result (Lemma \ref{l1}) provides the lower
bound for the Green's function with the pole at infinity
\begin{equation*}
G_{t}(z,\infty)\ge C\,\frac{d(z)}{|z|+1}\,, \quad z\in\O_{5t}, \quad 0<t\le t_0.
\end{equation*}
When $E$ is a finite set, the result can be improved $G_{t}(z,\infty)\ge C>0$, $z\in\O_{kt}$ with some $k=k(E)>1$. For
various estimates of the Green's functions and harmonic measures see, for instance \cite{Nev, GaMar, carl}.

\smallskip

Here is the main result of the paper.
\begin{theorem}\label{t2}
Let $E$ be an $r$-convex compact set with connected complement $\O=\ovl\C\backslash E$, and a subharmonic function $v$ satisfy
$\eqref{31}$. Let $\p$ be a positive, monotone
increasing and absolutely continuous function on $\R_+$, such that $\p_1(t):=t^{-1}\p(t)$ is monotone
increasing at the neighborhood of the origin, and
\begin{equation}\label{311}
\int_0^{1} \p_1'(t)\,\psi\left(\frac{t}5\right)dt+\int_{1}^\infty \p'(t)\,\psi\left(\frac{t}3\right)dt<\infty.
\end{equation}
Then the following Blaschke-type condition for the Riesz measure holds
\begin{equation}\label{312}
\int_{\O} \p(d(\z))\,\mu(d\z)\le C(E,\psi,\p) K_v.
\end{equation}
\end{theorem}
{\bf Remark 1}. Let $E$ be an $r$-convex compact set, $\widetilde\O$ its outer domain, that is, the unbounded component of $\O$.
Then the set ${\rm Pc}(E)=\ovl\C\backslash\widetilde\O$, aka the polynomial convex hull of $E$, is $r$-convex, and
$d(z)={\rm dist}(z,{\rm Pc}(E))$ for $z\in\tilde\O$. Given a subharmonic function $v$ \eqref{31}, its restriction
$\tilde v$ to $\widetilde\O$ satisfies the conditions of Theorem \ref{t2}, so the Blaschke-type condition \eqref{312}
holds with $\O$ replaced with $\widetilde\O$.
\smallskip

A typical example in Theorem \ref{t2} is $\psi(x)=x^{-q}$, $q>0$, where we can take
\begin{equation*}
\p(x)=x^{q+1/2}\,\left(\min\{x,1/x\}\right)^{\e+1/2}=
\left\{
  \begin{array}{ll}
    x^{q+1+\e}, & \hbox{$x\le1;$} \\
    x^{q-\e}, & \hbox{$x>1$.}
  \end{array}
\right.
\end{equation*}

\medskip

A special case of Theorem \ref{t2} with $v=\log|f|$, $f$ an analytic function, occurs in perturbation
theory in the study of discrete spectra for the Schatten--von Neumann perturbations of certain bounded linear operators. Given
a bounded linear operator $A_0$ on the Hilbert space $\cH$, and a compact operator $B$, the fundamental theorem of Weyl states
that the essential spectra of $A_0$ and $A=A_0+B$ agree, so the discrete eigenvalues of $A$ (the isolated eigenvalues of finite
algebraic multiplicity) can accumulate only at the joint essential spectrum. We want to gather some information on the rate of
accumulation under the stronger assumption that $B$ belongs to some Schatten--von Neumann operator ideal $\cS_q$, $1\le q<\infty$,
that is, if $\|B\|^q_{\cS_q}:=\sum_n s_n^q(B)<\infty$, $s_n(B)$ are the singular values of $B$. Under the rate of accumulation
we mean the inequalities of the form
\begin{equation}\label{disspec}
\sum_{\l\in\s_d(A)} d^p(\l)\le C\,\|B\|_{\cS_q}^q, \qquad d(\l):={\rm dist}(\l,\s(A_0))
\end{equation}
for some $p=p(q)$, $\s(T)$ ($\s_d(T)$) is the spectrum (discrete spectrum) of an operator $T$.

Kato \cite{kat} proved \eqref{disspec} for self-adjoint $A_0$ and $B\in\cS_q$ with $p=q\ge1$, $C=1$. Recently Hansmann \cite{han12}
obtained the same result for a self-adjoint $A_0$ and an arbitrary $B\in\cS_q$ with $p=q>1$ and the explicit (in a sense) constant
$C=C_q$.

For more general classes of operators \eqref{disspec} is shown to be true for both $A_0$ and $B$ normal with $p=q\ge2$, $C=1$
\cite{bol}, for all three $A_0$, $B$, $A$ normal with $p=q\ge1$, $C=1$ \cite{bhda}, and for $A_0$ normal, an arbitrary
$B\in\cS_q$ with $p=q\ge1$, $C=1$, under additional assumption that $\s(A_0)$ is a {\it convex} set \cite{han11}.

We apply Theorem \ref{t2} for the study of the rate of accumulation for an arbitrary $B\in\cS_q$ and operators $A_0$ (in general,
non-normal) with the $r$-convex spectrum and the growth of the resolvent governed by the distance to the spectrum (see precise
conditions (i)-(iii) in Section 5). The corresponding bound looks as follows
\begin{equation*}
\sum_{\l\in\s_d(A)} \Phi\lp d(\l)\rp\le C\, \|B\|_{\cS_q}^q,
\end{equation*}
$\Phi$ is a continuous function on $\R_+$, $\Phi(0)=0$. The result is illustrated with several examples.

\section{$r$-convexity}

Given an $r$-convex set $E$, it is in general hard enough to compute its radius of convexity. In some simple instances we can
work out this problem.

\begin{proposition}\label{p2}
Let $E=\{a,b,c\}$ be a $3$-point set in a general position, $R=R(abc)$ be the circumradius of the triangle $\De=\De(abc)$. Then
$r_0(E)=R(abc)$. Let $E$ be a compact subset of a circle $\partial B(y,\rho)$, $|E|\ge3$, then $r_0(E)=\rho$.
\end{proposition}
\begin{proof} Let us recall some known facts from elementary planar geometry.
A triangle $\De(abc)$ is always viewed as an open planar set.

\noindent
1. Given a triangle $\De(abc)$, an open disk $B(x,r)$ with $r>R$ such that
two vertices (say, $a$ and $b$) lie on its boundary, and $c\notin B(x,r)$, is uniquely determined.
We call it the $r$-disk and denote $(ab)_r$. If $r=R$ and $\De$ is acute, the $R$-disks $(ab)_R$, $(ac)_R$, and $(bc)_R$ are by definition
the reflections of the circumdisk $B_\De$ through the sides of $\De$. If $\De$ is non-acute, and $c$ is the vertex at the largest angle,
then $(ac)_R$ and $(bc)_R$ are defined as above, and $(ab)_R=B_\De$.

\noindent
2. If $\De$ is acute, then the circles $\partial(ab)_R$, $\partial(ac)_R$, and $\partial(bc)_R$ meet at one point in $\De$, precisely,
the orthocenter of $\De$, see, e.g., \cite[Problem 5.9]{Pra}. If $\De$ is non-acute, and $c$ is the vertex at the largest angle,
the circles $\partial(ab)_R$, $\partial(ac)_R$, and $\partial(bc)_R$ meet at $c$, and there is a circular triangle with one vertex
at $c$, which lies in $\De\backslash \lp\ovl{(ac)_R}\cup\ovl{(bc)_R}\rp$.

\noindent
3. For $r\ge R$ let $[ab]_r$ be the segment of the disk $(ab)_r$ with vertices $a$ and $b$, which intersects $\De$. Then for $R\le r_1<r_2$ we have
$[ab]_{r_2}\subset [ab]_{r_1}$, and the inclusion is proper.

It is clear that $r_0(E)\ge R$, so we wish to show that the complement to $E$ cannot be covered with open disks
of radius $r>R$ which avoid points $a$, $b$, and $c$. Since $\De$ is convex, we can restrict our attention
to the points of $\De$. Assume on the contrary that each point $x\in\De$ belongs to such disk. Then $x$ belongs to one of the three segments
from 3, so $\De\subset([ab]_r\cup[ac]_r\cup[bc]_r)$. But by 2 and 3 the latter union can not cover all $\De$.
Contradiction completes the proof of the first statement.

As far as the second statement goes, the set $E$ is clearly $r$-convex for $r\le \rho$. For $r>\rho$, as it has just been proved,
any 3-point set $E_1=\{a,b,c\}\subset E$ is not $r$-convex, and ${\rm conv}_r(E_1)$ contains points from $\De(abc)\subset B(y,\rho)$. Since
${\rm conv}_r(E_1)\subset {\rm conv}_r(E)$ for $E_1\subset E$, $E$ cannot be $r$-convex either, as claimed.
\end{proof}

{\bf Remark}. Given a triangle $\De(abc)$ with the circumradius $R$, let $E\subset\partial\De$ be a compact set, which contains all vertices
$a$, $b$, and $c$. It follows from the above proof and monotonicity of the $r$-convex hull, that for $r>R$ the intersection of
${\rm conv}_r(E)$ and $\De$ is nonempty.

\medskip

To extend the above result, let us say that a compact set $E$ has finite global curvature if
\begin{equation}\label{ecur}
r_g(E):=\inf\{R(abc)\}>0,
\end{equation}
where infimum is taken over all possible triangles with vertices in $E$. Clearly, \eqref{ecur} holds for finite sets.
When $E$ is a Jordan rectifiable curve, the value $r_g^{-1}(E)$ is known as the global curvature of $E$, see \cite{schmos}.

\begin{proposition}\label{mcur}
Each compact set $E$ with finite global curvature is $r$-convex, and
$$ r_0(E)=r_g(E). $$
\end{proposition}
\begin{proof}
Assume that for some $r>r_g(E)$ the set $E$ is $r$-convex. Take a
triangle $\Delta(a_0b_0c_0)$ so that $r>R(a_0b_0c_0)$. As it was shown in
the proof of Proposition \ref{p2}, the $r$-convex hull of this
triangle (and the more so, the $r$-convex hull of $E$ itself)
contains points of $\Delta(a_0b_0c_0)$. But it is easily seen from \eqref{ecur}
that $E$ has an empty interior. The contradiction shows that
$r_0(E)\le r_g(E)$.

This leaves only the converse inequality to be accounted for. We
show that each point $z\in\C\backslash E$ can be covered with a
disk $B\subset\C\backslash E$ of the radius at least $r_g(E)$.

Define
\begin{equation}\label{maxrad}
\rho_z:=\sup\{r:z\in B(x,r)\subset\C\backslash E\}.
\end{equation}
The compactness argument shows that there is a disk $B(x_z,\rho_z)$
$$ z\in B(x_z,\rho_z)\subset\C\backslash E. $$
If $\partial B(x_z,\rho_z)\cap E$ contains at least 3 different points, then $\rho_z\ge r_g(E)$, as needed.
Assume that $\partial B(x_z,\rho_z)\cap E=\{\z_1\}$, or $\partial B(x_z,\rho_z)\cap E=\{\z_1,\z_2\}$, and the points $\z_1,\z_2$
do not belong to a diameter of the circle. Then we can shift the disk in an appropriate direction
(perpendicular to the interval $[\z_1,\z_2]$ towards the center of the circle), and inflate it a bit to obtain
a bigger disk with the same property, which contradicts maximality of $\rho_z$ \eqref{maxrad}.

Hence we can focus upon the case $x_z=\rho_z$, $\partial B(\rho_z,\rho_z)\cap E=\{a,b\}$, $a=0$, $b=2\rho_z$
(after an affine transformation of the plane). Let $G=\{z: 0\le\Re z\le 2\rho_z\}$.

Assume first that there is a sequence $c_n=x_n+iy_n\in G\cap E$ with
$c_n\to a$ or $c_n\to b$ (with not loss of generality let the first relation hold). We want to show that now
$R(abc_n)\to\rho_z$. To this end, we apply an explicit formula for the
circumradius $R(z_1z_2z_3)$, suggested in \cite{mel}
\begin{equation}\label{meln}
 R^{-2}(z_1z_2z_3)=\sum_{\pi} \frac1{(z_{\pi(1)}-z_{\pi(2)})\ovl{(z_{\pi(1)}-z_{\pi(3)})}}=
\frac{4\Im^2\,(z_{1}-z_{2})\ovl{(z_{2}-z_{3})}}{|(z_{1}-z_{2})(z_{1}-z_{3})(z_{2}-z_{3})|^2}\,,
\end{equation}
where the sum is taken over all permutations of $\{1,2,3\}$. Since $c_n\in G\backslash B(\rho_z,\rho_z)$, $c_n\to 0$,
it is easy to see that $x_n/y_n\to 0$ as $n\to\infty$. It follows from \eqref{meln} that
\begin{equation*}
\begin{split}
R^{-2}(abc_n) &=\frac{16\rho_z^2\,y^2_n}{4\rho_z^2|2\rho_z-x_n-iy_n|^2(x_n^2+y_n^2)}=\frac{4y^2_n}{((2\rho_z-x_n)^2+y_n^2)(x_n^2+y_n^2)} \\
&=\frac4{((2\rho_z-x_n)^2+y_n^2)(x_n^2/y_n^2+1)}\to \frac1{\rho_z^2}, \end{split}\end{equation*}
as claimed. Hence now $\rho_z\ge r_g(E)$.

Otherwise, the disk can be shifted and inflated, as above,
which contradicts maximality of $\rho_z$. The proof is complete.
\end{proof}

\begin{proposition}\label{curve}
Each $C^2$-smooth Jordan curve $($arc$)$ has finite global curvature.
\end{proposition}
\begin{proof}
Assume on the contrary, that $r_g(\G)=0$, $\G$ is the $C^2$-smooth Jordan curve or arc. Then there is a sequence of triangles
$\De(a_nb_nc_n)$ with $R(a_nb_nc_n)\to 0$ as $n\to\infty$. By taking subsequences, if needed, we have $a_n\to a\in\G$, and so
$b_n,c_n\to a$, $n\to\infty$.

On the other hand, we will show that
\begin{equation}\label{cur1}
\lim_{n\to\infty} R^{-1}(a_nb_nc_n)=\tau(a)<\infty,
\end{equation}
$\tau(a)$ is the curvature of $\G$ at $a$, which will lead to contradiction.
Let
$$ \G=\{z(t)=x(t)+iy(t)\}, \quad (a_nb_nc_n)=(z(t_1)z(t_2)z(t_3)), \quad a=z(0). $$
We apply again \eqref{meln}
$$ \Im\,(z(t_1)-z(t_2))\ovl{(z(t_2)-z(t_3))}= (y(t_1)-y(t_2))(x(t_2)-x(t_3))-(x(t_1)-x(t_2))(y(t_2)-y(t_3)), $$
so
\begin{equation*}
\begin{split}
\frac{\Im\,(z(t_1)-z(t_2))\ovl{(z(t_2)-z(t_3))}}
{(t_{1}-t_{2})(t_{2}-t_{3})(t_{1}-t_{3})} &= \frac{[t_1t_2]_y\,[t_2t_3]_x-[t_1t_2]_x\,[t_2t_3]_y}{t_1-t_3} \\
&=[t_1t_2t_3]_y\,[t_2t_3]_x-[t_1t_2t_3]_x\,[t_2t_3]_y,
\end{split}\end{equation*}
where
$$ [t_it_k]_f:=\frac{f(t_i)-f(t_k)}{t_i-t_k}, \qquad
   [t_1t_2t_3]_f:=\frac{[t_1t_2]_f-[t_2t_3]_f}{t_1-t_3} $$
are divided differences of the first and second order, respectively. The limit relation below is one of
the basic properties of divided differences
$$ \lim_{t_i\to0} [t_1t_2t_3]_f=\frac12\,f''(0), $$
provided $f$ is a $C^2$-smooth function at the origin. Hence
$$ \lim_{t_i\to0} \frac{\Im\,(z(t_1)-z(t_2))\ovl{(z(t_2)-z(t_3))}}
{(t_{1}-t_{2})(t_{2}-t_{3})(t_{1}-t_{3})}=\frac{y''(0)x'(0)-x''(0)y'(0)}2\,, $$
and finally by \eqref{meln} and the definition of the curvature
$$ \lim_{t_i\to0} R^{-1}(z(t_1)z(t_2)z(t_3))=\frac{|y''(0)x'(0)-x''(0)y'(0)|}{|z'(0)|^3}\,. $$
The latter is \eqref{cur1}, as claimed.
\end{proof}

{\bf Remark}. Proposition \ref{curve} is a particular case of a much more sophisticated result
\cite[Theorem 1, (iii)]{schmos}, which claims that $\G$ has finite global curvature if and only if its
arc length parametrization $\tau(s)$ is smooth, and $\tau'$ satisfies the Lipschitz condition with
Lipschitz constant $r_g^{-1}(E)$.

\smallskip

Yet another example arises in the theory of elliptic equations in domains with non-smooth boundaries \cite{ad}.

\begin{definition}\label{bc}
A planar domain $\O$ with the boundary $\partial\O$ is said to satisfy the uniform ball condition
if there is $r>0$ so that for each $x\in\partial\O$ there is a ball $B$ of radius $r$ with the properties
$$ B\subset \O,  \qquad  x\in\partial B. $$
Let $\G$ be a Jordan curve, $\C\backslash\G=\O_i\cup\O_o$, interior and exterior domains of $\G$.
We say that $\G$ is BC-curve if both $\O_i$ and $\O_o$ satisfy the uniform ball condition.
A Jordan arc $\g$ is BC-arc if there is a BC-curve $\G\supset\g$.
\end{definition}

It is easy to see that if $E$ is an $r$-convex compact set then $\O=\C\backslash E$ satisfies the uniform ball condition. Indeed,
let $x\in\partial\O$. There is a sequence of points $z_n\in\O$ so that $z_n\to x$ as $n\to\infty$.
Take the corresponding sequence of disks $B_n$ of radius $r$, $z_n\in B_n\subset\O$. Then a certain subsequence of
$B_n$ converges to $B$ from definition \ref{bc}.

\begin{proposition}\label{bcr}
A Jordan curve (arc) is $r$-convex if and only if it is a BC-curve (arc).
\end{proposition}
\begin{proof} Due to the above remark we need to show that each BC-curve (arc) is $r$ convex.
Let $z\in\O_i$, and $d(z)=d(z,\G)<r$. Take $\z\in\G$ with $|z-\z|=d(z)$ and the ``supporting'' disks $B_i$ and $B_0$ of radius
$r$ at the point $\z$ as in definition \ref{bc}. Since $B_i\subset\O_i$, $B_o\subset\O_o$, the disks touch each other at $\z$.
The disk $B(z,d(z))\subset\O_i$ passes through $\z$, hence it is necessarily contained in $B_i$ and touches it at $\z$.
So $z\in B_i$, as needed. The argument for $z\in\O_o$ is the same.

As for the BC arc $\g$, take the BC curve $\G\supset\g$. Let $z\in\G\backslash\g$. Then the inner supporting disk $B_i$ at $z$
is disjoint with $\g$, so it can be shifted appropriately so that $z$ belongs to the shifted one, which is still disjoint with $\g$.
\end{proof}

A simple example $E=\{i/n\}\cup\{1/n\}\cup\{0\}$, $n=1,2,\ldots$, displays the set such that $E$ is not $r$-convex, but
$\C\backslash E$ satisfies the uniform ball condition.
\medskip

Given a compact set $E$  consider the unbounded open set $\O_t:=\{z\in\C: \ d(z)>t\}$, $t\ge0$. It is clear that $\{\O_t\}$ forms a
monotone decreasing family of sets.

Let $\Th_t$ be the unbounded component of $\O_t$, $\Th_t\subseteq\O_t$. It is not hard to manufacture a compact set $E$ so that $\Th_0=\O_0$,
i.e., $\C\backslash E$ is connected, but $\Th_t\not=\O_t$ for all $t>0$.
We show that this is not the case for $r$-convex sets, and the situation is stable for small enough $t$.

\begin{theorem}\label{t1}
Let $E$ be an $r$-convex compact set, $r>0$, and $\Th_0=\O_0$. Then there is $t_0=t_0(E)$ such that $0<t_0\le r/4$, and
$\Th_t=\O_t$ for $0\le t\le t_0$.
\end{theorem}
\begin{proof} Denote by $S=S(E):=\max_{\z\in E} |\z|$, and assume that $r\le S(E)$.

The proof is split into several steps.

\underline{Step 1}. The set $\hat\O:=\O_r\cap B(0,2S)$ is relatively compact, so it has a finite $r/2$-net $Z=\{z_j\}_{j=1}^N$,
$$ Z\subset \hat\O, \qquad {\rm dist}(z',Z)\le\frac{r}2<r, \quad \forall z'\in\hat\O. $$
Since $\O_0=\Th_0=\C\backslash E$ is connected, we can find pathes $\G_j: \ [0,\infty)\to \Th_0$ with
$$ \G_j(0)=z_j, \qquad \G_j(\tau)\to\infty, \quad \tau\to\infty; \quad j=1,2,\ldots,N. $$
We put $\d:=\frac12\,\min_{j,\tau} {\rm dist}(\G_j(\tau),E)>0$, so that $\G_j\subset\Th_\d$, $j=1,2,\ldots,N$.

\underline{Step 2}. Let $z'\in\hat\O$, there is $z_k\in Z$, $1\le k\le N$ such that $|z'-z_k|<r$. In other words,
$z'\in B(z_k,r)$, $z_k\in B(z',r)$. Put
$$ B_1:=B(z_k,r)\bigcup B(z',r), \qquad \{\xi_{\pm}\}:=\partial B(z_k,r)\bigcap \partial B(z',r). $$
Since both $z'$ and $z_k$ are in $\O_r$, then the closure $\ovl{B}_1\subset\O_0$, so
\begin{equation*}
\begin{split}
{\rm dist}([z',z_k],E) &> {\rm dist}([z',z_k],\partial B_1)={\rm dist}([z',z_k],\{\xi_{\pm}\}) \\
&=\sqrt{r^2-\frac{|z'-z_k|^2}4}>\frac{\sqrt{3}}2\,r.
\end{split}
\end{equation*}

Now, take $t_0:=\min(\d,r/4)$, so $[z',z_k]\subset\O_t$ for $t\le t_0$, and $\G_k\subset\O_t$  by Step 1. Hence
$[z',z_k]\cup\G_k\subset\O_t$, and as the set in the left hand side is a path from $z'$ to infinity, we conclude
$$ [z',z_k]\cup\G_k\subset\Th_t \Rightarrow z'\in\Th_t, \quad \forall t\le t_0. $$
Clearly, $B^c(0,2S)\subset\Th_t$ for such $t$, so, finally, $\O_r\subset\Th_t$, $t\le t_0$.

\underline{Step 3}. Assume that for some $\eta$, $0<\eta\le t_0$ the statement is wrong, so $\O_\eta$ has a bounded
component $D$, $D\cap\Th_\eta=\emptyset$. We want to show that
\begin{equation}\label{01}
d(z)\le \sqrt{2}\,\eta, \qquad \forall z\in D.
\end{equation}
Let $z\in D$. Note that $d(z)\le r$, for otherwise $z\in\O_r\subset\Th_\eta$ by Step 2, and hence $z\in D\cap\Th_\eta$,
that is impossible. By the definition of $r$-convexity $z\in B(z',r)\subset\O_0$, so $z'\in\O_r\subset\Th_\eta$, and,
in particular, $z'\not=z$. Hence the segment $[z',z]$ meets the boundary $\partial D$, so there is a point $\z\in[z',z]$
with $d(\z)=\eta$, and we conclude that
\begin{equation}\label{02}
{\rm dist}([z',z], E)\le\eta.
\end{equation}

\smallskip

Let us examine the mutual configuration of two disks, $B(z',r)$ and $B(z,d(z))$, each of which belongs to $\O_0$. As the circle
$\partial B(z,d(z))$ contains points from $E$, it is clear that the closed disk $\ovl{B(z,d(z))}$ cannot lie inside $B(z',r)$.
Hence either the smaller disk $B(z,d(z))$ touches the bigger one from within, and in this case the touching point $\xi\in E$
(which implies $d(z)={\rm dist}([z',z], E)$, and \eqref{01} follows from \eqref{02}), or the disks have a proper intersection. Denote
$$ B_2:=B(z',r)\bigcup B(z,d(z))\subset\O_0, \quad \{\xi_{\pm}\}:=\partial B(z',r)\bigcap \partial B(z,d(z)). $$
Then
\begin{equation}\label{03}
{\rm dist}([z',z], E)\ge {\rm dist}([z',z], \partial B_2)={\rm dist}([z',z], \xi_+)={\rm dist}([z',z], \xi_-)\ge h,
\end{equation}
where $h$ is the length of the altitude from the vertex $\xi_+$ in the triangle $\De(z',z,\xi_+)$. If this altitude crosses the side $[z',z]$ then
$$ \sqrt{r^2-h^2}+\sqrt{d^2(z)-h^2}=|z-z'|<r, \quad 2h^2>d^2(z)+2\sqrt{(r^2-h^2)(d^2(z)-h^2)}>d^2(z), $$
so $d(z)<\sqrt{2}h$, and \eqref{01} follows from \eqref{02},\eqref{03}. If the altitude crosses the extension of the side $[z',z]$,
one has $d(z)={\rm dist}([z',z], E)$, and \eqref{01} holds again.

The inclusion $z\in D\subset\O_\eta$ means $d(z)>\eta$, so we come to the following two-sided bound
\begin{equation}\label{04}
\eta<r':=\sup_{z\in D}d(z)\le\sqrt2\,\eta.
\end{equation}
Let $\{z_n\}\subset D$ so that $d(z_n)\to r'$. We can assume $z_n\to z_0$, and hence there is a point $z_0\in D$ with
$\eta<d(z_0)=r'\le\sqrt2\,\eta$.

\underline{Step 4}. We show here that there is a triangle $\De(abc)$ of the circumradius $R(\De)<4\eta$ such that
\begin{equation}\label{05}
\De(abc)\bigcap E=\emptyset, \qquad E_1:=\ovl{\De(abc)}\bigcap E\supset\{a,b,c\}.
\end{equation}

Take the point $z_0$ from Step 3 and consider the disk $B(z_0,r')$. Its boundary has nonempty intersection with $E$. If the circle
$\partial B(z_0,r')$ contains 3 different points from $E$, then in view of \eqref{04} we are done. Assume that
$\partial B(z_0,r')\cap E=\{\z_1\}$, or $\partial B(z_0,r')\cap E=\{\z_1,\z_2\}$, and the points $\z_1,\z_2$ do not belong to a
diameter of the circle. The same argument as in the proof of Proposition \ref{mcur} shows that such configurations cannot occur.

Therefore we can focus upon the case $z_0=0$, $\partial B(z_0,r')\cap E=\{a,b\}$, $a=ir'$, $b=-ir'$ (after an appropriate affine
transformation of the plane). Put
$$ G:=\{z=x+iy:\ 0<x\le r', \ |y|\le r'\}, $$
then $G\cap E\not=\emptyset$, since otherwise the circle could be shifted to the right to have $\ovl{B(z_0',r')}\cap E=\emptyset$,
which, as we have already seen, contradicts the maximality of $r'$. Let $h$ be the least (in absolute value) nonzero number such
that the triangle $\De(abc_h)$, $c_h=r'+ih$, contains points from $E$. The number $h$ exists since by the assumption the point
$C_0=r'\not\in E$, and $0<|h|\le r'$. Clearly, such points from $E$ belong to the side $ac_h$ for $h>0$ ($bc_h$ for $h<0$).
If we choose the point $c\in E$ on the corresponding side, then \eqref{05} holds.
The triangle $\De(abc)$ is either acute or rectangular. For its sides we have by \eqref{04}
\begin{equation}\label{06}
M:=\max(|ab|,|ac|,|bc|)\le\sqrt5 r'\le\sqrt{10}\eta,
\end{equation}
and by the known upper bound for the circumradius of such triangle $R(\De)\le M<4\eta$.

\underline{Step 5}. The choice of $t_0=\min(\d,r/4)$ implies $R(\De)<4\eta\le4t_0\le r$. By Proposition~\ref{p2}
(see Remark after its proof) the set $E_1$ \eqref{05} is not $r$-convex, and ${\rm conv}_r(E_1)\cap\De(abc)\not=\emptyset$. Hence,
${\rm conv}_r(E)\cap\De(abc)\not=\emptyset$, which contradicts the $r$-convexity of $E$.

To remove the assumption $r\le S(E)$, note that if $r>S(E)$, then $E$ is $r_1$-convex with $r_1=S(E)$. So for the value $t_0$ one has
$0<t_0\le r_1/4< r/4$, as needed. The proof is complete.
\end{proof}

Note that for $E=\{\z_1,\ldots,\z_N\}$ the result is obvious with
\begin{equation}\label{f1}
0\le t\le t_1(E):=\frac12\,\d(E), \qquad \d(E):=\min_{i\not=k}|\z_i-\z_k|.
\end{equation}

\section{Lower bounds for Green's functions}

In what follows we denote by $C=C(E)$ different positive constants which depend only on $E$, and particular values of which are immaterial.

We will be dealing with domains $\O=\ovl\C\backslash E$, $E$ a compact set in $\C$.

\begin{definition}
 The Green's function for the domain $\O$ is a map $G_\O:\ \O\times\O\to (-\infty,\infty]$, such that for each $w\in\O$
\begin{itemize}
  \item [(i)] $G_\O(\cdot,w)$ is harmonic on $\O\backslash\{w\}$, and bounded from above and below outside each neighborhood of $w$;
  \item [(ii)] $G_\O(w,w)=\infty$, and as $z\to w$,
  \begin{eqnarray*}
  G_\O(z,w) &=& -\log|z-w|+O(1), \ \ \ w\not = \infty, \\
  G_\O(z,w) &=& \log|z|+O(1), \ \ \ w = \infty;
  \end{eqnarray*}
  \item [(iii)] $G_\O(z,w)\to 0$, as $z\to\z$, nearly everywhere on $\z\in\partial \O$.
\end{itemize}
\end{definition}

Let us list some basic properties of the Green's functions in the way we need them later on (cf., e.g., \cite[Section 4.4]{Ran}):
\begin{enumerate}
  \item If $\partial\O$ is non-polar, then there exist a unique Green's function $G_\O$ for $\O$;
  \item $G_\O(z,w)=G_\O(w,z)>0$, moreover, if $\O'$ is a relatively compact in $\ovl\C$ open subset of $\O$,
  then $\min_{z,w\in\O'} G_\O(z,w)=C(\O,\O')>0$;
  \item If $\O'\subset\O''$ be domains in $\ovl\C$ with non-polar boundaries, then
      $$ G_{\O'}(z,w)\le G_{\O''}(z,w), \quad z,w\in\O'. $$
\end{enumerate}

The notion of the Harnack distance proves useful for our reasoning (see \cite[pp. 14--15]{Ran}).

\begin{definition}
Let $D$ be a domain in $\ovl\C$. Given $z,w\in D$, the Harnack distance between $z$ and $w$ is the smallest number
$\tau_D(z,w)$ so that for every positive harmonic function $h$ on $D$,
$$  \tau_D^{-1}(z,w) h(w)\le h(z)\le \tau_D(z,w) h(w). $$
\end{definition}
It is known that
\begin{enumerate}
  \item $\tau_D(z,w)=\tau_D(w,z)\ge 1$, $\tau_D(z,z)=1$;
  \item $\tau_D(z_1,z_3)\le \tau_D(z_1,z_2)\tau_D(z_2,z_3)$, $z_1,z_2,z_3\in D$;
  \item $\tau_D$ is a continuous function of both variables, in particular, if $D_1$ is a relatively compact in with respect to topology
  $\ovl\C$, open subset of $D$, then $\max_{z,w\in D_1}\tau_D(z,w)=C(D,D_1)<\infty$.
\end{enumerate}

\medskip

Given a compact set $E$, we remind the notation $\O_t=\{z\in\ovl\C:\ d(z)>t\}$ (we view $\O_t$ as an open subset of $\ovl\C$).
If $E$ is an $r$-convex compact set with connected complement, then by Theorem \ref{t1} $\O_t$ is a subdomain of $\ovl\C$ for small
enough $t$. Its boundary $\partial\O_t=\{z:d(z)=t\}$ is non-polar, so the Green's function $G_t$ for $\O_t$ exists and is unique.

The main technical tool is the following lower bound for $G_t(z,\infty)$.

\begin{lemma}\label{l1}
Let $E$ be an $r$-convex compact set with connected complement $\O=\ovl\C\backslash E$. Then for $0<t\le t_0$
\begin{equation}\label{10}
G_{t/5}(z,\infty)\ge C\,\frac{d(z)}{|z|+1}\,, \quad z\in\O_t.
\end{equation}
\end{lemma}
\begin{proof} By Theorem \ref{t1} $\O_t=\Th_t$ for $0<t\le t_0$.

Assume first that $d(z)>t_0/2$, so $z\in\O_{t_0/2}$. By properties $(2)$, $(3)$
$$ G_{t/5}(z,\infty)\ge G_{t_0/5}(z,\infty)\ge C>0, \qquad z\in\O_{t_0/2}. $$
Since $d(z)\le |z|+S\le (S+1)(|z|+1)$, $S=\max_{\z\in E} |\z|$, \eqref{10} follows.

For the rest of the proof we will assume $d(z)\le t_0/2$, and $z\in\O_{t}$, so $t<d(z)\le t_0/2$. By $r$-convexity $z\in B(z',r)\subset\O$,
and the following chain of inequalities can be easily checked
\begin{equation}\label{11}
r>|z-z'|\ge d(z')-d(z)\ge r-\frac{t_0}2>\frac{t_0}2\ge d(z)>t.
\end{equation}
Denote
$$ r_1:=|z-z'|-\frac{t}5, \qquad r_2:=d(z)-\frac{t}5, $$
so $4t/5<r_2<r_1<|z-z'|$. It follows from \eqref{11} that the disks $B(z',r_1)$ and $B(z,r_2)$ satisfy
\begin{enumerate}
  \item [(a)] $B(z',r_1)\cup B(z,r_2)\subset\O_{t/5}$;
  \item [(b)] $z'\notin B(z,r_2)$, \ \ $z\notin B(z',r_1)$;
  \item [(c)] Since
  $$ |z-z'|<r_1+\frac{r_2}4=|z-z'|-\frac{t}5+\frac{d(z)-t/5}4, $$ then $B(z',r_1)\cap B(z,r_2/4)\not=\emptyset$.
\end{enumerate}

Put $L:=\partial B(z',r_1)\cap B(z,r_2/2)$, the arc of the circle $\partial B(z',r_1)$ inside $B(z,r_2/2)$.
A simple argument from the plane geometry shows that property (c) implies the lower bound for the length of $L$: $|L|>r_2/2$.

\smallskip

We proceed with the bounds for the Green's functions. By properties (a) and (b) the function $G_{t/5}(\cdot,z)$ is harmonic and
positive in the disk $B(z',r_1)$. As $r_1<|z-z'|<r$, the Mean Value Theorem provides
$$ G_{t/5}(z',z)=\frac1{2\pi r_1}\,\int_{\partial B(z',r_1)} G_{t/5}(\z,z)\,m(d\z)\ge \frac1{2\pi r}\,\int_{L} G_{t/5}(\z,z)\,m(d\z). $$
Since $B(z,r_2)\subset\O_{t/5}$, and the Green's function increases with the domain, we have
$$ G_{t/5}(u,v)\ge G_{B(z,r_2)}(u,v), \qquad u,v\in B\left(z,r_2\right). $$
The latter can be computed explicitly
$$ G_{B(z,r_2)}(z,v)=\log\left|\frac{r_2}{v-z}\right|\ge\log 2, \qquad v\in B\left(z,\frac{r_2}2\right). $$
Hence $G_{t/5}(z,\z)\ge\log 2$ for $\z\in L$, so taking into account $r_2=d(z)-t/5>4d(z)/5$, we come to the lower bound
\begin{equation}\label{12}
G_{t/5}(z',z)\ge \frac{\log 2}{2\pi r}\,|L|>\frac{\log 2}{4\pi r}\,r_2>\frac{\log 2}{5\pi r}\,d(z).
\end{equation}

To pass from $z'$ to $\infty$, we invoke the Harnack distance.
Put $D=\O_{t_0}$, a domain in $\ovl\C$ which depends only on $E$, $D\subset\O_{t/5}$, and consider a function $h_{t,z}(\z):=G_{t/5}(z,\z)$.
It is clear that $z\not\in D$ (by the assumption $d(z)\le t_0/2$), so $h_{t,z}$ is positive and harmonic in $D$. Next, $t_0<r$ yields $\O_r$
is a relatively compact subset of $D$, so $z'\in D$, and by the definition of the Harnack distance
$$ \tau_D^{-1}(z',\infty)\,G_{t/5}(z,z')\le G_{t/5}(z,\infty). $$
By property (3) of the Harnack distance $\min_{z'\in\O_r}\tau_D^{-1}(z',\infty)=C>0$, and hence by \eqref{12}
$$ G_{t/5}(z,\infty)\ge C d(z)\ge C\frac{d(z)}{|z|+1}\,, \qquad z\in\O_t, $$
as claimed. The proof is complete.      \end{proof}

{\bf Remark 2}. Assume that $E$ is a non-polar $r$-convex compact set with connected complement. Then the Green's function
$G=G_0$ exists and unique, and it easily follows from Lemma \ref{l1} that
$$ G(z,\infty)\ge C\,\frac{d(z)}{|z|+1}\,, \quad z\in\O. $$

For the similar bounds for Green's functions of a bounded domain with $C^2$ boundary see \cite[formula (2.8)]{st}.

\section{Proof of the main result and its consequences}

We go over to subharmonic functions and their Riesz measures. Let $\cD$ be a domain of $\ovl\C$ such that its boundary
$\partial\cD$ is non-polar, and let $v$ be a subharmonic function on $\cD$, $v\equiv\kern-2ex/ -\infty$,
which has a harmonic majorant on $\cD$. Let $\mu=1/2\pi\, \De v$ be its Riesz measure.
By the fundamental Riesz decomposition theorem (RDT)
(cf., e.g., \cite[Theorem 4.5.4]{Ran})
\begin{equation*}%\label{30}
v(z)=u(z)-\int_{\cD}G(z,\z)\mu(d\z), \qquad z\in\cD,
\end{equation*}
$u$ is the least harmonic majorant on $\cD$, $G$ is the Green's function of $\cD$.

We apply this result for subharmonic functions on $\O=\ovl\C\backslash E$, $E$ is an $r$-convex compact set, with $\cD=\O_t$
for $t\le t_0$ from Lemma \ref{l1}, so its boundary is non-polar, and $G=G_t$.
As for the subharmonic on $\O$ function $v$, we assume that it is subject to the growth
and normalization conditions \eqref{31}, $\psi$ is a positive and monotone decreasing function on $\R_+$, $\psi\to +\infty$ as $t\to 0+$.
Hence $v$ has a harmonic majorant on $\cD$, and so
\begin{equation}\label{30}
v(z)=u_t(z)-\int_{\O_t}G_t(z,\z)\mu(d\z), \qquad z\in\O_t.
\end{equation}

{\it Proof of Theorem \ref{t2}}.
By \eqref{31} $v$ is bounded above on $\O_t$, and
\begin{equation}\label{313}
v(z)\le K_v \,\psi(t), \qquad z\in\O_t, \quad t>0, \end{equation}
so the least harmonic majorant $u_t$ does exists, with the same bound \eqref{313}. \eqref{30} for $z=\infty$ gives
\begin{equation}\label{321}
\int_{\O_t}G_t(\infty,\z)\mu(d\z)\le K_v \,\psi(t), \quad t>0.
\end{equation}

Next, write the left hand side of \eqref{312} as
$$ \int_\O \p(d(\z))\,\mu(d\z)=\int_\O \p(d(\z))\,\mu_1(d\z) + \int_{B^c(0,6S+1)} \p(d(\z))\,\mu(d\z)=I_1+I_2, $$
where $\mu_1$ is the restriction of $\mu$ to $B(0,6S+1)\backslash E$, so $\mu_1$ has a bounded support.

We begin with the bound for $I_1$. Put
$$ H_1(t,\mu_1):=\int_{\O_t}d(\z)\,\mu_1(d\z). $$
Since $d(\z)\le 6S+1+S=7S+1$ on the support of $\mu_1$, we have $H_1=0$ for $t\ge 7S+1$. We apply the so-called ``layer cake representation''
theorem (LCR) \cite[Theorem 1.13]{lilo}, which is a combination of the change of variables and the integration by parts, in the form
\begin{equation*}%\label{34}
\int_{\O} \p(d(\z))\,\mu_1(d\z)=\int_{\O} \p_1(d(\z))d(\z)\,\mu_1(d\z)  =\int_0^{7S+1} \p_1'(t)H_1(t,\mu_1)\,dt,
\end{equation*}
so
$$ I_1=\int_0^{7S+1} \p_1'(t)\,H_1(t,\mu_1)\,dt=\int_0^{t_0}\p_1'(t)\,H_1(t,\mu_1)\,dt+
\int_{t_0}^{7S+1}\p_1'(t)\,H_1(t,\mu_1)\,dt=I_{11}+I_{12}. $$
For $0<t\le t_0$ Lemma \ref{l1} combined with \eqref{321} gives
\begin{equation}\label{35}
H_1(t,\mu_1)\le C\int_{\O_t} G_{t/5}(\z,\infty)\,\mu_1(d\z)\le C\int_{\O_{t/5}} G_{t/5}(\z,\infty)\,\mu(d\z)
\le CK_v\,\psi\left(\frac{t}5\right),
\end{equation}
so
\begin{equation}\label{36}
I_{11}=\int_0^{t_0}\p_1'(t)\,H_1(t,\mu_1)\,dt\le CK_v\,\int_0^{t_0} \p_1'(t)\,\psi\left(\frac{t}5\right)dt.
\end{equation}
As for $I_{12}$, we have by \eqref{35}
$$ H_1(t,\mu_1)\le H_1(t_0,\mu_1)\le CK_v\,\psi\left(\frac{t_0}5\right), $$
and so
\begin{equation}\label{37}
I_{12}=\int_{t_0}^{7S+1}\p_1'(t)\,H_1(t,\mu_1)\,dt\le CK_v\,\psi\left(\frac{t_0}5\right)\,\int_{t_0}^{7S+1}\p_1'(t)dt=C K_v.
\end{equation}

The bound for $I_2$ is standard, and has nothing to do with the subtle Lemma \ref{l1}. Given $t\ge 5S+1$ we put
$$ R_t:=\frac23\,(t-S)\ge \frac23(4S+1), \qquad R_t-S=\frac{2t-5S}3\ge\frac{t}3, $$
and apply again the RDT in the form
$$ v(z)=\tilde u(z)-\int_{|\z|>R_t} \tilde G(z,\z)\,\mu(d\z), \qquad |z|>R_t, $$
$\tilde G$ is the Green's function of the domain $\{\z:|\z|>R_t\}$, $\tilde u$ the least harmonic majorant on this domain. Since
$d(z)\ge R_t-S$ for $|z|>R_t$, the assumptions on $v$ imply
$$ \tilde u(z)\le K_v\,\psi(R_t-S), \qquad |z|>R_t, $$
and so, as at the beginning of the proof,
\begin{equation}\label{37}
\int_{|\z|>R_t} \tilde G(\infty,\z)\,\mu(d\z)\le K_v\,\psi(R_t-S)\le K_v\,\psi\left(\frac{t}3\right).
\end{equation}
The function $\tilde G(\infty,\z)$ is known explicitly, $\tilde G(\infty,\z)=\log|\z|-\log|R_t|$, so
we conclude by \eqref{37}
$$ \log\frac32\,\int\limits_{|\z|>\frac32\,R_t}\,\mu(d\z)\le \int\limits_{|\z|>\frac32\,R_t} \log\left|\frac{\z}{R_t}\right|\,\mu(d\z)
\le \int\limits_{|\z|>R_t} \log\left|\frac{\z}{R_t}\right|\,\mu(d\z)\le K_v\,\psi\left(\frac{t}3\right). $$

Next, note that
$$ \left\{\z: |\z|>\frac32 R_t\right\}\supset \left\{\z: d(\z)>\frac32 R_t+S\right\}=\O_t, $$
so
\begin{equation}\label{0101}
 H(t,\mu):=\int_{\O_t}\mu(d\z)\le CK_v\,\psi\left(\frac{t}3\right)\,. \end{equation}
We apply the LCR theorem to $I_2$, having in mind $\{\z: |\z|>6S+1\}\subset \O_{5S+1}$, so by \eqref{0101}
\begin{equation}\label{38}
I_2\le \int_{\O_{5S+1}} \p(d(\z))\,\mu(d\z)=\int_{5S+1}^\infty \p'(t)\,H(t,\mu)dt\le CK_v\,
\int_{5S+1}^\infty \p'(t)\psi\left(\frac{t}3\right)\,dt.
\end{equation}
Theorem \ref{t2} now follows from \eqref{36}, \eqref{37} and \eqref{38}.
  \hfill$\Box$

\begin{corollary}\label{cor1}
Let for a subharmonic function $v$ $\eqref{31}$ holds with $\psi(t)=t^{-q}$, $q>0$. Then for each $\e>0$
\begin{equation}\label{coroll1}
\int_{\O} \p(d(\z))\,\mu(d\z)\le C(E,q,\e) K_v
\end{equation}
with
\begin{equation*}
\p(x)=x^{q+1/2}\,\left(\min\{x,1/x\}\right)^{\e+1/2}=
\left\{
  \begin{array}{ll}
    x^{q+1+\e}, & \hbox{$x\le1;$} \\
    x^{q-\e}, & \hbox{$x>1$.}
  \end{array}
\right.
\end{equation*}
\end{corollary}

In some instances, in addition to the hypothesis of Theorem \ref{t2}, the support of the Riesz measure $\mu$ appears
to be bounded. Such situation occurs when $v=\log|f|$, $f$ is an analytic function on $\O$ with $f(\infty)=1$ (see Section 5).
Now only the first term in \eqref{311} matters, so we come to the following
\begin{corollary}\label{cor2}
In addition to the assumptions of Theorem $\ref{t2}$, let ${\rm supp}\,\mu\subset B(0,R_\mu)$, and, instead of \eqref{311},
$$ \int_0 \p_1'(t)\,\psi\left(\frac{t}5\right)\,dt<\infty. $$
Then
\begin{equation*}
\int_{\O} \p(d(\z))\,\mu(d\z)\le C(E,\psi,\p,R_\mu) K_v.
\end{equation*}
\end{corollary}

\medskip

Consider the case of finite sets $E$, where the bound for the Green's function in \eqref{10} and the main result can be refined.
We formulate it for the special bound as in
Corollary \ref{cor1}, although the general case of Theorem \ref{t2} can be handled in the same fashion.

\begin{theorem}\label{t3}
Let $E=\{\z_1,\ldots,\z_N\}$ be a finite set, $v$ be a subharmonic function on $\O=\ovl\C\backslash E$ so that
\begin{equation}\label{381}
v(z)\le\frac{K_v}{d^q(z)}\,, \qquad q>0, \quad z\in\O,
\end{equation}
and $v(\infty)=0$. Then there is $k=k(E)>1$ such that
\begin{equation}\label{393}
G_t(z,\infty)>\frac{\log 2}{N}>0, \qquad z\in\O_{kt},
\end{equation}
and for each $\e>0$
\begin{equation*}
 \int_{\O} \p(d(\z))\,\mu(d\z)\le C(E,q,\e) K_v, \qquad
\p(x)=
\left\{
  \begin{array}{ll}
    x^{q+\e}, & \hbox{$x\le1;$} \\
    x^{q-\e}, & \hbox{$x>1$.}
  \end{array}
\right.
\end{equation*}
If, in addition, ${\rm supp}\,\mu$ is bounded then
\begin{equation}\label{382}
\int_{\O} d^{q+\e}(\z)\,\mu(d\z)<\infty.
\end{equation}
\end{theorem}
\begin{proof}
Put
\begin{equation}\label{39}
m_j:=\prod_{i\not=j}|\z_i-\z_j|, \qquad C:=2^{N-1}\,\max_j m_j.
\end{equation}
The function
\begin{equation}\label{391}
v_t(z):=\frac1{N}\left(\sum_{j=1}^N\log|z-\z_j|-\log t-\log C\right) \end{equation}
is subharmonic on $\C$ (and harmonic on $\O$), and $v_t(z)=\log|z|+O(1)$, as $z\to\infty$. For $t\le t_1$ in \eqref{f1}, on each
circle $|z-\z_n|=t$, $n=1,2,\ldots,N$ one has
$$ v_t(z)=\frac1{N}\left(\sum_{j\not=n}\log|z-\z_j|-\log C\right), $$
and since $|z-\z_j|\le |z-\z_n|+|\z_n-\z_j|=t+|\z_n-\z_j|$, then
$$ v_t(z)\le \frac1{N}\left(\sum_{j\not=n}\log(|\z_j-\z_n|+t)-\log C\right)\le \frac1{N}\left((N-1)\log 2+
\sum_{j\not=n}\log|\z_n-\z_j|-\log C\right)\le0 $$
in view of the choice of $C$. Hence $u_t(z)=v_t(z)-G_t(z,\infty)$ is subharmonic on $\O_t$,
$$ \limsup_{z\to\z}u_t(z)\le0, \quad \z\in\partial\O_t, \quad \limsup_{z\to\infty}\frac{u_t(z)}{\log|z|}=0, $$
so by the Phragmen--Lindel\"of principle \cite[Corollary 2.3.3]{Ran} $u_t\le0$, or
\begin{equation}\label{392}
v_t(z)\le G_t(z,\infty), \qquad z\in\O_t.
\end{equation}

On the other hand, put
$$ k=k(E):=1+2C\,\left(\frac2{\d(E)}\right)^{N-1}>1, $$
$\d(E)$ is defined in \eqref{f1}, and assume that $t\le t_2(E):=k^{-1}t_1$. For $z\in\O_{k t}$ we have
\begin{equation*}
\min_i|z-\z_i| =|z-\z_l|>kt, \qquad \min_{i\not=l}|z-\z_i|\ge \d(E)-kt,
\end{equation*}
so
\begin{equation*}
\begin{split}
v_t(z) &=\frac1{N}\left(\sum_{j\not=l}\log|z-\z_j|+\log|z-\z_l|-\log t-\log C\right) \\ &>\frac1{N}\left((N-1)\log(\d(E)-kt)
+\log kt-\log t-\log C\right) \\
&\ge \frac1{N}\left( (N-1)\log\frac{\d(E)}2+\log k-\log C\right)>\frac{\log 2}{N}\,,
\end{split}
\end{equation*}
by the choice of $k$ and $C$. Finally,
\begin{equation*}
G_t(z,\infty)\ge v_t(z)>\frac{\log 2}{N}>0, \qquad z\in\O_{kt},
\end{equation*}
as needed.

The rest of the proof goes along the same line of reasoning as one in Theorem \ref{t2}, by using the ``layer cake representation'',
with Lemma \ref{l1} replaced with \eqref{393}.
\end{proof}

\medskip

To show that Corollary 1 and Theorem \ref{t3} are optimal in a sense, we proceed with the following simple result.

\begin{lemma}\label{le2}
Let $E$ be an arbitrary compact set, which does not split the plane, $D$ be a relatively compact
$($in the sense of\ $\ovl\C)$ subdomain of $\O=\ovl\C\backslash E$, and $\infty\in D$.
Let $v$ be a subharmonic and continuous $($in the sense of\ $\ovl\C)$, nonnegative function on
$\O$. Then the least harmonic majorant $u$ for $D$ exists, and
\begin{equation}\label{opt1}
v_{min}:=\min_{\z\in\partial D} v(\z)\le u(z)\le \max_{\z\in\partial D} v(\z)=:v_{max}, \quad z\in D.
\end{equation}
\end{lemma}
\begin{proof} By the assumption, $v$ is nonnegative and bounded on $D$, so the least harmonic majorant exists, and
it is nonnegative and bounded.

To prove the right inequality, note that $v$ is continuous on $\ovl D$, and so
$$ \limsup_{z\to\z}v(z)=v(\z)\le v_{max}. $$
By the Maximum Principle $v\le v_{max}$, so $u\le v_{max}$.

To prove the left inequality, note that
$$ \liminf_{z\to\z}u(z)\ge \liminf_{z\to\z}v(z)=v(\z)\ge v_{min}. $$
Put $V=-u+v_{min}$, the harmonic and bounded function on $D$, and
$\limsup_{z\to\z}V(z)\le 0$, $\z\in\partial D$.
Again, by the Maximum Principle, $V\le 0$ in $D$, as needed.
The proof is complete.
\end{proof}

For the class of subharmonic functions $v$ $\eqref{31}$ with $\psi(t)=t^{-q}$, $q>0$, there is an obvious extremal element
$\hat v(z)=d^{-q}(z)$. This function is subharmonic and continuous on $\O$, and it is quite natural to expect that it provides
certain opposite results (divergence of integrals in \eqref{coroll1}).

Let us apply Lemma \ref{le2} to $\hat v$. By the RDT
$$ 0=\hat v(\infty)=\hat u(\infty)-\int_D G_D(z,\infty)\,\hat\mu(dz), \quad \hat\mu=\frac1{2\pi}\Delta \hat v,$$
and so by Lemma \ref{le2}
\begin{equation}\label{opt2}
[\max_{\z\in\partial D} d(\z)]^{-q}\le \int_D G_D(z,\infty)\,\hat\mu(dz)\le [\min_{\z\in\partial D} d(\z)]^{-q}\,.
\end{equation}

Two types of domains $D$ are of particular interest.

\noindent
1. Let, as above in Section 2, $\Theta_t$ be the unbounded component of the set $\O_t=\{z:\ d(z)>~t\}$. Then $d(\z)=t$ on
$\partial \Theta_t$, so by \eqref{opt2}
\begin{equation}\label{opt3}
\int_{\Theta_t} G_{\Theta_t}(z,\infty)\,\hat\mu(dz)=t^{-q}\,.
\end{equation}

\noindent
2. Let $D=D_t=\{|z|>t\}$, $t>S=\max_{\z\in E} |\z|$. Then for $|z|\ge t$
\begin{equation}\label{opt4}
\frac{t-S}{t}\,|z|\le d(z)\le |z|+S,
\end{equation}
$G_{D_t}(z,\infty)=\log\frac{|z|}t$, and \eqref{opt2} takes the form
\begin{equation}\label{opt41}
(t+S)^{-q}\le \int_{D_t} \log\frac{|z|}t\,\hat\mu(dz)\le (t-S)^{-q}\,.
\end{equation}

Let us mention two important consequences of \eqref{opt41}. First, let $t>\tau>S$, then
\begin{equation}\label{opt5}
\begin{split}
\int_{D_t} \hat\mu(dz) &\le \left(\log\frac{t}{\tau}\right)^{-1}\,\int_{D_t} \log\frac{|z|}{\tau}\,\hat\mu(dz)\le
\left(\log\frac{t}{\tau}\right)^{-1}\,\int_{D_\tau} \log\frac{|z|}{\tau}\,\hat\mu(dz) \\
&\le \left(\log\frac{t}{\tau}\right)^{-1}(\tau-S)^{-q}<\infty.
\end{split}
\end{equation}
Next,
\begin{equation}\label{opt6}
\int_{D_t} \log|z|\,\hat\mu(dz)\le (t-S)^{-q}+\log t\,\int_{D_t} \hat\mu(dz)<\infty.
\end{equation}

We show now that Corollary 1 is false for the function $\hat v$ and $\e<0$.

\begin{theorem}\label{thopt1}
Let $E$ be an arbitrary compact set, which does not split the plane, $\hat v(z)=d^{-q}(z)$, $q>0$.
Then for each $\e>0$
\begin{equation}\label{opt7}
I_{\pm}:=\int_{\O} d^{q\pm\,\e}(z)\,\hat\mu(dz)=+\infty.
\end{equation}
\end{theorem}
\begin{proof} Put $M:=B(0,S+1)\backslash E=B(0,S+1)\bigcap\O$. We actually prove that
$$ \int_{D_{S+1}} d^{q+\,\e}(z)\,\hat\mu(dz)=\int_{M} d^{q-\,\e}(z)\,\hat\mu(dz)=+\infty, \qquad D_{S+1}=\{|z|>S+1\}.
$$
Let us begin with $I_+$. By \eqref{opt4} with $t=S+1$ we have for $|z|\ge S+1$
$$ d^{q+\,\e}(z)\ge \frac{|z|^{q+\,\e}}{(S+1)^{q+\,\e}}\ge C_1(E,q,\e)\,|z|^q\log|z|, $$
so that
\begin{equation}\label{opt8}
\int_{D_{S+1}} d^{q+\,\e}(z)\,\hat\mu(dz)\ge C_1(E,q,\e)\int_{D_{S+1}} |z|^q\log|z|\,\hat\mu(dz).
\end{equation}
Let $\s(dz)=\log|z|\,\hat\mu(dz)$ restricted to $D_{S+1}$. The LCR theorem gives
\begin{equation*}
\begin{split}
\int_{D_{S+1}} |z|^q\,\s(dz) &= q\int_0^\infty t^{q-1}\,dt\,\int_{D_{t}} \log|z|\,\hat\mu(dz) \\
&= (S+1)^q\,\int_{D_{S+1}} \log|z|\,\hat\mu(dz)+q\int_{S+1}^\infty t^{q-1}\,dt\,\int_{D_{t}} \log|z|\,\hat\mu(dz),
\end{split}
\end{equation*}
so by \eqref{opt8}
\begin{equation*}
\int_{D_{S+1}} d^{q+\,\e}(z)\,\hat\mu(dz)\ge C_2(E,q,\e)\,\int_{S+1}^\infty t^{q-1}\,dt\,\int_{D_{t}} \log|z|\,\hat\mu(dz).
\end{equation*}
But $G_{D_t}(z,\infty)=\log|z|-\log t<\log|z|$, and it follows from \eqref{opt41} that
\begin{equation*}
\int_{D_{t}} \log|z|\,\hat\mu(dz)\ge \int_{D_{t}} G_{D_t}(z,\infty)\,\hat\mu(dz)\ge (t+S)^{-q},
\end{equation*}
which implies
$$ I_+\ge \int_{D_{S+1}} d^{q+\,\e}(z)\,\hat\mu(dz)=+\infty, $$
as claimed.

The domain $\Theta_x$ plays a key role in estimating $I_-$. Let $z\in\Theta_x$, then for every $z_0\in E$ one has
$$ |z-z_0|\ge d(z)>x, \qquad \frac{|z-z_0|}{x}>1, $$
so the function $h(z)=\log\frac{|z-z_0|}{x}$ is harmonic on $\Theta_x$,
$$ h(z)\ge0, \quad z\in \ovl{\Theta}_x; \qquad h(z)=\log|z|+O(1), \quad z\to\infty. $$
Hence by the Maximum Principle
\begin{equation}\label{opt9}
\log\frac{|z-z_0|}{x}-G_{\Theta_x}(z,\infty)\ge 0, \quad z\in\Theta_x.
\end{equation}

Denote $M_x:=B(0,S+1)\bigcap\O_x$, $N_x:=B(0,S+1)\bigcap\Theta_x\subset M_x$. If $z\in N_x$ and $x<1$, then \eqref{opt9} implies
\begin{equation}\label{opt10}
G_{\Theta_x}(z,\infty)<\log\frac{2S+1}x<C_3(E,\e)\,x^{-\e}.
\end{equation}
We apply again LCR theorem to obtain
$$
\int_{M} d^{q-\,\e}(z)\,\hat\mu(dz) = (q-\e)\int_0^1 x^{q-\e-1}dx\,\int_{M_x}\hat\mu(dz)
\ge (q-\e)\int_0^1 x^{q-\e-1}dx\,\int_{N_x}\hat\mu(dz). $$
By \eqref{opt10}
$$
\int_{M} d^{q-\,\e}(z)\,\hat\mu(dz)
\ge C_4(E,q,\e)\,\int_0^1 x^{q-1}dx\,\int_{N_x} G_{\Theta_x}(z,\infty)\,\hat\mu(dz). $$
Next, obviously $N_x=\Theta_x\backslash{\ovl{D}_{S+1}}=\Theta_x\backslash(\Theta_x\cap\ovl{D}_{S+1})$, and so
$$ \int_{N_x} G_{\Theta_x}(z,\infty)\,\hat\mu(dz)=\int_{\Theta_x} G_{\Theta_x}(z,\infty)\,\hat\mu(dz)-
\int_{\Theta_x\cap\ovl{D}_{S+1}} G_{\Theta_x}(z,\infty)\,\hat\mu(dz). $$
The first integral in the right hand side is $x^{-q}$ due to \eqref{opt3}. As for the second one, we have
by \eqref{opt9}, \eqref{opt5} and \eqref{opt6}
\begin{equation*}
\begin{split}
\int_{\Theta_x\cap\ovl{D}_{S+1}} G_{\Theta_x}(z,\infty)\,\hat\mu(dz)
&\le \int_{\Theta_x\cap\ovl{D}_{S+1}} \log\frac{|z-z_0|}x\,\hat\mu(dz)\le \int_{\ovl{D}_{S+1}} \log\frac{2|z|}x\,\hat\mu(dz) \\
&= \int_{\ovl{D}_{S+1}} \log|z|\,\hat\mu(dz)+\log\frac2x\,\int_{\ovl{D}_{S+1}}\,\hat\mu(dz)\le C_5(E)\left(1+\log\frac2x\right).
\end{split}
\end{equation*}
Finally,
$$ \int_{N_x} G_{\Theta_x}(z,\infty)\,\hat\mu(dz)\ge x^{-q}-C_5(E)\left(1+\log\frac2x\right)\ge C_6(E)x^{-q} $$
for small enough $x$, and so
$$ I_-\ge \int_{M} d^{q-\,\e}(z)\,\hat\mu(dz)=+\infty. $$
The proof is complete.
\end{proof}

We can write \eqref{opt7} as (compare with Corollary 1)
\begin{equation*}
\int_{\O} \hat\p(d(z))\,\hat\mu(dz)=+\infty, \qquad \hat\p(x)=
\left\{
  \begin{array}{ll}
    x^{q-\,\e}, & \hbox{$x\le1;$} \\
    x^{q+\,\e}, & \hbox{$x>1$.}
  \end{array}
\right.
\end{equation*}

It turns out that for particular sets $E$ and the function $\hat v$ Corollary 1 is false even for $\e=0$.

{\bf Example}. Let $E_0=[0,1]$, $v_0(z)=d^{-2}(z, E_0)$, $\mu_0=\frac1{2\pi}\Delta v_0$. By Corollary 1
\begin{equation*}
\int_{\O_0} \p_0(d(\z))\,\mu_0(d\z)<\infty, \qquad
\p_0(x)=
\left\{
  \begin{array}{ll}
    x^{3+\e}, & \hbox{$x\le1$,} \\
    x^{2-\e}, & \hbox{$x>1$,}
  \end{array}
\right. \quad \forall\e>0.
\end{equation*}
We can compute the Riesz measure explicitly. Indeed,
now $\C=\C_1\cup\C_2\cup\C_3$, where
$$ \C_1=\{z: 0\le x\le 1, \ y\not=0\}, \quad \C_2=\{z: x<0\}, \quad \C_3=\{z: x>1\}, \quad z=x+iy. $$
We apply the well-known equality $\Delta|F|^2=4|F'|^2$, $F$ is an analytic function, so
\begin{equation*}
v_0(z)=\left\{
         \begin{array}{ll}
           y^{-2}, & \hbox{$z\in\C_1$,} \\
           |z|^{-2}, & \hbox{$z\in\C_2$,} \\
           |z-1|^{-2}, & \hbox{$z\in\C_3$,}
         \end{array}
       \right.
\qquad
\Delta v_0=\left\{
             \begin{array}{ll}
               6y^{-4}, & \hbox{$z\in\C_1$,} \\
               4|z|^{-4}, & \hbox{$z\in\C_2$,} \\
               4|z-1|^{-4}, & \hbox{$z\in\C_3$.}
             \end{array}
           \right.
\end{equation*}
We have for $p>0$
$$ \int_{\O_0} d^p(z)\,\mu_0(dz)=\sum_{j=1}^3 \int_{\C_j} d^p(z)\,\mu_0(dz). $$
The first integral
$$ I_1:= \int_{\C_1} d^p(z)\,\mu_0(dz)=12\int_0^1 dx\int_0^\infty \frac{dy}{y^{4-p}}=+\infty $$
for $p=3$. The second one
$$ I_2:= \int_{\C_2} d^p(z)\,\mu_0(dz)=8\int_{-\infty}^0 dx\int_0^\infty \frac{dy}{(x^2+y^2)^{2-p/2}}=+\infty $$
for $p=2$. The computation for $I_3$ is similar.

\medskip

We complete the section with the converse result for analytic functions (cf. \cite{fg12}).

\begin{proposition}\label{p3}
Let $E$ be a compact subset of $\C$, $Z=\{z_n\}$ a sequence of points in $\O=\C\backslash E$ so that
$$ K:=\sum_{n\ge1} d^q(z_n)<\infty, \qquad q\ge1. $$
Then there is an analytic on $\O$ function $f$ with the zero set $Z(f)=Z$, $f(\infty)=1$, such that
\begin{equation}\label{394}
\log|f(z)|\le \frac{C_q K}{d^q(z)}.
\end{equation}
\end{proposition}
\begin{proof}
We begin with the well known Weierstrass prime factor of order $p=0,1,\ldots$
$$ W(z,p)=(1-z)\exp\left(\sum_{k=1}^p \frac{z^k}{k}\right), \quad p\ge1, \qquad W(z,0)=1-z, $$
and its bounds
\begin{equation}\label{395}
|W(z,p)-1|\le |z|^{p+1}, \qquad |z|\le1,
\end{equation}
\begin{equation}\label{396}
\log|W(z,p)|\le A_p|z|^{p}, \quad |z|\ge\frac13, \qquad A_p=3e(2+\log(p+1)).
\end{equation}

Denote by $e_n\in E$ one of the closest points to $z_n$, i.e., $d(z_n)=|z_n-e_n|$. Put
$$ f(z):=\prod_{n\ge1} W(u_n(z),p), \quad u_n(z)=\frac{z_n-e_n}{z-e_n}\,, $$
$p=0,1,\ldots$ is taken from $q-1\le p<q$, and write
$$ f(z)=\Pi_1(z)\cdot\Pi_2(z), \quad \Pi_j(z)=\prod_{n\in\L_j} W(u_n(z),p), \quad j=1,2, $$
where
$$ \L_1=\L_1(z)=\{n:|u_n(z)|\le1\}, \qquad \L_2=\L_2(z)=\{n:|u_n(z)|>1\}. $$
Since $u_n(z)\to 0$ for each $z\in\O$, the product $\Pi_2$ is finite. By \eqref{395}
$$ \sum_{n\in\L_1} |W(u_n(z),p)-1|\le \sum_{n\in\L_1} |u_n(z)|^{p+1}\le \sum_{n\in\L_1}|u_n(z)|^{q}\le \frac{K}{d^q(z)}\,, $$
so the product $\Pi_1$ converges absolutely and uniformly in $\O$. Besides,
$$ \log|\Pi_1(z)|\le \sum_{n\in\L_1} |W(u_n(z),p)-1|\le \frac{K}{d^q(z)}\,. $$
As for the second product, by \eqref{396}
$$ \log|\Pi_2(z)|\le A_p \sum_{n\in\L_2}|u_n(z)|^p\le A_p \sum_{n\in\L_2}|u_n(z)|^q\le \frac{A_p K}{d^q(z)}\,, $$
which proves \eqref{394}. The equality $Z(f)=Z$ is obvious by the construction.
\end{proof}

\section{Applications in perturbation theory of linear operators}

Recall some rudiments from the spectral theory of linear operators on the Hilbert space, related to the structure of
the spectrum (see, e.g., \cite[Section IV.5.6]{Kato}). A bounded linear operator $T$ on the infinite-dimensional Hilbert
space $\cH$ is said to be a Fredholm operator if its kernel and cokernel are both finite-dimensional subspaces. A complex
number $\l$ lies in the essential spectrum $\s_{ess}(T)$ of operator $T$ if $T-\l$ is not a Fredholm operator. The essential
spectrum is known to be a nonempty closed subset of the spectrum $\s(T)$, and its complement $\cF(T)=\C\backslash\s_{ess}(T)$
is called a Fredholm domain of $T$ (it is not necessarily connected, though). Clearly, the resolvent set
$\rho(T)=\C\backslash\s(T)\subset\cF(T)$.

The set of all isolated eigenvalues of finite algebraic multiplicity is referred to as the discrete spectrum $\s_d(T)=\{\l_j\}$,
each eigenvalue is counted according to its algebraic multiplicity. Although $\s_{ess}(T)\cap\s_d(T)=\emptyset$, the whole spectrum
is not in general exhausted by their union. Indeed, write
$$ \cF(T)=\bigcup_{j\ge0} \cF_j(T), $$
$\cF_j(T)$ are the connected components of $\cF(T)$, $\cF_0$ is the unbounded component (the outer domain). Then either
$\cF_j\subset\s(T)$, or $\cF_j\cap\s(T)\subset\s_d(T)$ (the latter always occurs for $j=0$). So $\cF(T)$ is connected
($\cF(T)=\cF_0(T)$) implies (the union is disjoint)
\begin{equation}\label{40}
\s(T)=\s_{ess}(T)\,\dot{\bigcup}\,\s_d(T).
\end{equation}

The fundamental theorem of Weyl \cite[Theorem IV.5.35]{Kato} is an outstanding result in perturbation theory. Its version for
bounded operators states that the essential spectrum is stable under compact perturbations, that is, for any bounded operator
$A_0$ and compact operator $B$
\begin{equation}\label{41}
\s_{ess}(A)=\s_{ess}(A_0), \qquad A=A_0+B.
\end{equation}
Under certain conditions (see below) relation \eqref{40} holds for the spectrum $\s(A)$ of the perturbed operator as well,
and all accumulation points of $\s_d(A)$ belong to $\s_{ess}(A_0)$. We are aimed here at finding the quantitative rate of
convergence in the form
$$ \sum_{\l\in\s_d(A)} \Phi(d(\l))\le C\|B\|_{\cS_q}^q, \qquad d(\l)={\rm dist}(\l,\s(A_0)), \quad q\ge1, $$
provided $B\in\cS_q$, the Schatten--von Neumann operator ideal.

Our main assumptions on the unperturbed operator $A_0$ are as follows:
\begin{itemize}
  \item[(i)]  $\s_{ess}(A_0)$ does not split the plane;
  \item[(ii)] $\s(A_0)$ is an $r$-convex compact set;
  \item[(iii)] The resolvent $R(\l,A_0)=(A_0-\l)^{-1}$ is subject to the bound
\begin{equation}\label{43}
\|R(\l,A_0)\|\le \Psi(d(\l)), \qquad \l\not\in\s(A_0),
\end{equation}
$\Psi$ is a monotone decreasing from $+\infty$ to $0$ function on $\R_+$.
\end{itemize}

Note that conditions (i) and (ii) are certainly fulfilled whenever $\s(A_0)\subset\R$ or $\s(A_0)\subset\T$ and $\s(A_0)\not=\T$.
As for condition (iii), it is not really a restriction, as one can put
$$ \Psi(x)=\sup\{\|R(\l,A_0)\|: \ \ d(\l)\ge x\}. $$
However such choice of $\Psi$ is very much implicit. There is a variety of operators, for which \eqref{43} holds with
explicit function $\Psi$. Among them, e.g., hyponormal operators \cite[Theorem 3.10.2]{Put} and spectral in the sense
of Dunford operators of finite degree \cite{dsh3} (with $\Psi(x)=x^{-s}$, $s>0$). For normal operators $A_0$ the equality
prevails in \eqref{43} with $\Psi(x)=x^{-1}$. Another typical example is (see \cite{nik76, Gil})
$$ \Psi(x)=\frac{C_1}{x}\,\exp\lp\frac{C_2}{x^2}\rp. $$

\smallskip

A key analytic tool in perturbation theory is the (regularized) perturbation determinant
$$ g_q(\l):=\det{}_{\lceil q\rceil}(I+BR(\l,A_0)), \quad B=A-A_0\in\cS_q, $$
$\lceil q\rceil=\min\{n\in\N:n\ge q\}$, thanks to the following properties (see \cite[Section XI.9]{dsh2},
\cite[Section IV.3]{gk}, \cite{simideal})
\begin{enumerate}
  \item $g_q$ is analytic on $\ovl\C\backslash\s(A_0)$, $g_q(\infty)=1$;
  \item $\l$ is the zero of $g_q$ of multiplicity $k$ if and only if $\l\in\s_d(A)\backslash\s(A_0)$ with algebraic multiplicity $k$;
  \item $\log|g_q(\l)|\le C_q\,\|B\|^q_{\cS_q}\,\|R(\l,A_0)\|^q\,$, $\l\in\ovl\C\backslash\s(A_0)$.
\end{enumerate}

We are in a position to present the main spectral consequences of Theorem \ref{t2} and Corollary~\ref{cor2}.
\begin{theorem}\label{t4}
Given a bounded linear operator $A_0$ subject to conditions $(i)-(iii)$, and $B\in\cS_q$, $q\ge1$, let $\Phi$ be a positive and
absolutely continuous function on $[0,\infty)$ such that $\Phi_1(t)=t^{-1}\Phi(t)$ is monotone increasing at the neighborhood
of the origin, and
\begin{equation}\label{45}
\int_0^{1} \Phi_1'(t)\,\Psi^q\left(\frac{t}5\right)dt+\int_{1}^\infty \Phi'(t)\,\Psi^q\left(\frac{t}3\right)dt<\infty.
\end{equation}
Then
\begin{equation}\label{46}
\sum_{\l\in\s_d(A)} \Phi\lp d(\l)\rp\le C(\s(A_0), \Psi,\Phi,q)\,\|B\|_{\cS_q}^q.
\end{equation}
\end{theorem}
\begin{proof} Since $\s_{ess}(A_0)=\s_{ess}(A)$ does not split the plane, we see that
\begin{equation}\label{461}
 \s(A_0)=\s_{ess}(A_0)\,\dot{\bigcup}\,\s_d(A_0), \quad \s(A)=\s_{ess}(A)\,\dot{\bigcup}\,\s_d(A),
\end{equation}
and so both $\s(A_0)$ and $\s(A)$ do not split the plane.

We apply Theorem \ref{t2} with $E=\s(A_0)$ to the subharmonic function
$$ v(z)=\log|g_q(z)|, \qquad z\in \rho(A_0). $$
In view of property (3) of perturbation determinants, and condition (iii), inequality \eqref{31} holds with  $K_v=C_q\|B\|_{\cS_q}^q$
and $\psi=\Psi^q$. The Riesz measure is now a discrete and integer-valued measure supported on $Z(g_q)$, and $\mu\{\l\}$ equals the
multiplicity of the zero of $g_q$ at $\l$ (the algebraic multiplicity of the eigenvalue $\l(A)$). The only problem is that in \eqref{461}
$\s_d(A_0)$ is, generally speaking, nonempty, and, what is more to the point, the set $\s_d(A_0)\cap\s_d(A)$ can be nonempty as well,
and this part of $\s_d(A)$ is not controlled by the zero set of the perturbation determinant. \footnote{As a matter of fact,
the Weinstein--Aronszajn formula says that the order of zero (pole) of $g_p$ at the point $\l\in\s_1(A)$ equals
$\nu(\l(A))-\nu(\l(A_0))\in\Z$, the difference of algebraic multiplicities of the eigenvalue $\l$.} Anyway, since $\Phi(0)=0$,
Theorem \ref{t2} leads to
$$ \sum_{\l\in\s_d(A)} \Phi\lp d(\l)\rp= \sum_{\l\in\s_d(A)\backslash\s_d(A_0)} \Phi\lp d(\l)\rp\le
   C(\s(A_0), \Psi,\Phi,q)\,\|B\|_{\cS_q}^q,  $$
as claimed.
\end{proof}

{\bf Remark}. A question arises naturally, whether condition (i) can be relaxed to
\begin{itemize}
  \item[(i')] $\s(A_0)$ does not split the plane.
\end{itemize}
The answer is negative. Indeed, there are examples of operators $A_0$, $A$ so that
$$ \s_{ess}(A_0)=\partial\D, \quad \s(A_0)=\ovl\D, \qquad \s(A)=\partial\D\,\dot{\bigcup}\,\s_d(A), $$
and the portion of $\s_d(A)$ inside $\D$ is out of reach.

If $\s_{ess}(A_0)$ splits the plane, then (see Remark 1 after Theorem \ref{t2}) we can argue as above with the resolvent set $\rho(A_0)$
replaced by the outer domain $\cF_0(A_0)$, and end up with the inequality
$$ \sum_{\l\in\s_d(A)\cap\cF_0(A_0)} \Phi\lp d(\l)\rp\le
   C(\s(A_0), \Psi,\Phi,q)\,\|B\|_{\cS_q}^q. $$

\begin{corollary}\label{cor3}
In the assumptions of Theorem $\ref{t4}$
$$ \int_0 \Phi_1'(t)\,\Psi^q\left(\frac{t}5\right)dt<\infty \Rightarrow
\sum_{\l\in\s_d(A)} \Phi\lp d(\l)\rp\le C(\s(A_0),\Psi,\Phi,q,\|B\|)\,\|B\|^q_{\cS_q}. $$
\end{corollary}
\begin{proof}
In our setting the support of the Riesz measure is bounded, ${\rm supp}\,\mu\subset B(0,R_\mu)$, so Corollary \ref{cor2} applies.
It remains only to show that the value $R_\mu$ is controlled by the operator norm $\|B\|$. Indeed,
it is proved in \cite[Lemma 8.4.2]{Gil} that under condition \eqref{43}
$$ \max_{\z\in\s(A)} d(\z)\le x(\Psi, \|B\|^{-1}), $$
where $x(\Psi,a)$, $a>0$, is the largest solution of the equation $\Psi(x)=a$. So one can take
$$ R_\mu= \sup_{\l\in\s(A_0)}|\l|+x(\Psi,\|B\|^{-1}), $$ as needed.
\end{proof}

\smallskip

{\bf Example 1}. Let $A_0$ be a bounded linear operator with a real spectrum, $\s(A_0)\subset\R$, and condition \eqref{43} hold with
$\Psi(x)=x^{-p}$, $p>0$. Now both $\s_{ess}(A_0)$ and $\s(A_0)$ are compact subsets of the real line, so they are $r$-convex and do
not split the plane. So for $A=A_0+B$, $B\in\cS_q$, and each $\e>0$, the bounds
\begin{equation}\label{47}
\sum_{\l\in\s_d(A)} \Phi\lp d(\l)\rp\le
   C(\s(A_0),p,q,\e)\,\|B\|_{\cS_q}^q, \quad \Phi(x)=\left\{
  \begin{array}{ll}
    x^{pq+1+\e}, & \hbox{$x\le1;$} \\
    x^{pq-\e}, & \hbox{$x>1$,}
  \end{array}
\right.
\end{equation}
and
\begin{equation}\label{48}
\sum_{\l\in\s_d(A)} d^{pq+1+\e}(\l)<\infty
\end{equation}
hold. In particular, if $W$ is a bounded linear operator with imaginary component from $\cS_q$, relations \eqref{47} and \eqref{48}
are true with $p=1$,
$$ A_0=W_R=\frac{W+W^*}2, \qquad B=W_I=\frac{W-W^*}{2i}\,. $$

The stronger result for self-adjoint $A_0$ is in \cite{han12}. Its direct application to operators $A_0$ similar to self-adjoint
($A_0=T^{-1}A_1T$, $A_1=A_1^*$) would lead to the constant $C$ which appears on the right hand side and depends on the transform $T$.
In \eqref{47} this constant depends only on the {\it spectrum} of $A_0$.

{\bf Example 2}. The same argument works for unitary (or similar to unitary) operators $A_0$ such that there is $\z\in\T\cap\rho(A_0)$.
In particular, let $V$ be an $\cS_q$-quasiunitary operator, that is, $I-V^*V\in\cS_q$, and $\z\in\T\cap\rho(V)$. Then its Cayley
transform $W=i(\z+V)(\z-V)^{-1}$ satisfies
$$ W_I=(\bar\z-V^*)^{-1}\{I-V^*V\}(\z-V)^{-1}\in\cS_q. $$
Note that $W+i=2i\z(\z-A)^{-1}$ is invertible, and $V=\z(W+i)^{-1}(W-i)$. It is easy to see that
$$ V=U+B, \qquad B\in\cS_q, \quad U=\z(W_R+i)^{-1}(W_R-i) $$
is a unitary operator with $\s(U)\not=\T$, so the bound similar to \eqref{48} holds with $A_0=U$, $A=V$.

{\bf Example 3}. In the Hilbert space $L^2[0,1]$ consider an operator
$$ [Af](x)=a_0(x)f(x)+\int_0^1 K(x,y)f(y)\,dy $$
with the Hilbert--Schmidt kernel $K$, i.e., $K\in L^2([0,1]\times [0,1])$. We assume that the function $a_0$ is complex valued,
continuous on $[0,1]$, and the arc $\g=\{a_0(x): \ 0\le x\le1\}$ is Jordan and either a BC-arc (see Definition \ref{bc}) or has finite
global curvature (in particular, $C^2$-smooth). The multiplication operator $A_0f=a_0f$ is normal, and its spectrum $\s(A_0)=\g$
is the $r$-convex compact set with connected complement (see Propositions \ref{mcur} and \ref{bcr}). As in \eqref{47} we have
\begin{equation}\label{50}
\sum_{\l\in\s_d(A)} \Phi\lp d(\l)\rp\le
   C(\g,\e)\,\|K\|_{\cS_2}^2, \quad \Phi(x)=\left\{
  \begin{array}{ll}
    x^{3+\e}, & \hbox{$x\le1;$} \\
    x^{2-\e}, & \hbox{$x>1$.}
  \end{array}
\right.
\end{equation}

\noindent
{\bf Acknowledgement}. We thanks A. Eremenko for helpful remarks concerning $r$-convexity.

\bibliographystyle{amsplain}

\end{document}